\documentclass{article}

\usepackage{amsmath, amssymb, amsthm, amsfonts, bm}
\usepackage{enumerate}
\usepackage{commath}
\usepackage{mathrsfs}
\usepackage[retainorgcmds]{IEEEtrantools}
\usepackage{verbatim}
\numberwithin{equation}{section}
\newcommand{\ud}{\, \mathrm{d}}

\newtheorem{cthm}{Theorem}

\newtheorem{lemma}{Lemma}[section]
\theoremstyle{definition}
\newtheorem{definition}{Definition}[section]
\theoremstyle{remark}

\newcommand{\Rnum}[1]{\uppercase\expandafter{\romannumeral #1\relax}}

\newcommand{\mr}[1]{\mathrm{#1}}
\newcommand{\mb}[1]{\mathbb{#1}}
\newcommand{\mc}[1]{\mathcal{#1}}

\DeclareMathOperator{\mt}{\mathcal{M}_\mathcal{T}}

\title{Sharp and general estimates for the Bellman function of three integral variables related to the dyadic maximal operator}
\author{Anastasios D. Delis, Eleftherios N. Nikolidakis}
\date{}

\begin{document}
\maketitle
\footnotetext{ {\em E-mail addresses}: tdelis@math.uoa.gr, lefteris@math.uoc.gr}

\begin{abstract}
We compute the Bellman function of three integral variables associated to the dyadic maximal operator on a subset of its domain. Additionally, we provide an upper bound for the whole domain of its definition.
\end{abstract}

\section{Introduction}\label{sec:1}

It is well known that the dyadic maximal operator on $\mb R^n$ is a useful tool in analysis and is defined by
\begin{equation} \label{eq:1p1}
\mathcal{M}_d \phi(x) = \sup \left\{ \frac{1}{|Q|}\int_Q|\phi(y)|\,\mr dy: x\in Q,\ Q\subseteq \mb R^n\ \text{is a dyadic cube}\right\},
\end{equation}
for every $\phi\in L_\text{loc}^1(\mb R^n)$, where the dyadic cubes are those formed by the grids $2^{-N}\mb Z^n$, for $N=0, 1, 2,\ldots$.
It is also well known that it satisfies the following weak type (1,1) inequality
\begin{equation} \label{eq:1p2}
\left|\left\{x\in\mb R^n : \mathcal{M}_d \phi(x) > \lambda\right\}\right| \leq
\frac{1}{\lambda} \int_{\{\mathcal{M}_d\phi>\lambda\}}|\phi(y)|\,\mr dy,
\end{equation}
for every $\phi\in L^1(\mb R^n)$ and every $\lambda>0$, and which is easily proved to be best possible. Further refinements of \eqref{eq:1p2} can be seen
in \cite{10} and \cite{11}.

Then by using \eqref{eq:1p2} and the well known Doob's method it is not difficult to prove that the following $L^p$ inequality is also true
\begin{equation} \label{eq:1p3}
\|\mathcal{M}_d\phi\|_p \leq \frac{p}{p-1}\|\phi\|_p,
\end{equation}
for every $p>1$ and $\phi\in L^p(\mb R^n)$.
Inequality \eqref{eq:1p3} turns out to be best possible and its sharpness is proved in \cite{17} (for general martingales see \cite{1} and \cite{2}).

One way to study inequalities satisfied by maximal operators is by using the so called Bellman function technique. This approach was first introduced in the work of Nazarov and Treil, \cite{6}, where the authors defined the function  
\begin{multline}\label{eq:1p4}
B_p(f,F,L) =\\ \sup\left\{ \frac{1}{|Q|} \int_Q (\mathcal{M}_d \phi)^p : \frac{1}{|Q|}\int_Q\phi=f,\ \frac{1}{|Q|}\int_Q\phi^p=F, \sup_{R:Q\subseteq R}\frac{1}{|R|}\int_R\phi=L\right\},
\end{multline}
with $p>1$ (as an example they examine the case $p=2$), where $Q$ is a fixed dyadic cube and $\phi$ is non negative in $L^p(Q),$ $R$ runs over all dyadic cubes containing $Q$ and the variables $F, f, L$ satisfy $0\leq f\leq L,$ $f^p\leq F.$ Exploiting a certain "pseudoconcavity" inequality it satisfies, they construct
the function $4F-4fL+2L^2$ which has the same properties as \eqref{eq:1p5} and provides a good $L^p$ bound for the operator $\mathcal{M}_d$ (see \cite{6} for details).

The exact value of the above Bellman function was explicitly computed for the first time by Melas in \cite{4}. 
In fact this was done in the much more general setting of a non-atomic probability space $(X,\mu)$ 
equipped with a tree structure $\mc T$, which is similar to the structure of the dyadic subcubes of 
$[0,1]^n$ (see the definition in Section \ref{sec:2}). Then, the associated maximal operator is defined
by
\begin{equation} \label{eq:1p5}
\mt\phi(x) = \sup\left\{ \frac{1}{\mu(I)}\int_I |\phi|\,\mr d\mu: x\in I\in \mc T\right\},
\end{equation}
for every $\phi\in L^1(X,\mu)$.
Moreover \eqref{eq:1p2} and \eqref{eq:1p3} still hold in this setting and remain sharp. 

The Bellman function of two variables related to the above maximal operator is given by
\begin{equation} \label{eq:1p6}
B_{\mc T}^{(p)}(f,F) = \sup\left\{ \int_X(\mt\phi)^p\,\mr d\mu: \phi\geq 0,\ \int_X\phi\,\mr d\mu = f,\ \int_X\phi^p\,\mr d\mu = F\right\},
\end{equation}
where $0< f^p\leq F$.
In \cite{3} it is proved that
\begin{equation}\label{eq:1p7}
B_{\mc T}^{(p)}(f,F) = F\,\omega_p\!\left(\frac{f^p}{F}\right)^p,
\end{equation}
where $\omega_p : [0,1] \to \bigl[1,\frac{p}{p-1}\bigr]$, is defined by $\omega_p(z) = H_p^{-1}(z)$, 
and $H_p(z)$ is given by $H_p(z) = -(p-1)z^p + pz^{p-1}$. As a consequence $B_{\mc T}^{(p)}(f,F)$ does 
not depend on the tree $\mc T$.
The technique for the evaluation of \eqref{eq:1p6}, that is used in \cite{4}, is based on an effective 
linearization of the dyadic maximal operator that holds on an adequate class of functions, which is enough to describe the problem as is settled on 
\eqref{eq:1p6}. Using this result on suitable subsets of $X$ and several calculus arguments, the 
author also managed to precisely evaluate the corresponding to \eqref{eq:1p4} Bellman function in this context,
\begin{multline}\label{eq:1p8}
B_{\mc T}^{p}(f,F,L) = \sup\bigg\{\int_X (\max(\mathcal{M}_{\mc T} \phi, L)^p \ud \mu :\phi\geq0, \phi \in  L^p(X,\mu),\\ \int_X\phi \ud \mu=f, \int_X\phi^p \ud \mu=F, \bigg\}.
\end{multline}

We wish to refine \eqref{eq:1p3} even further. So we consider the $q$-norm, $1<q<p,$ of the function $\phi$ as fixed and aim to compute the corresponding Bellman function:
\begin{multline}\label{eq:1p9}
B_{\mc T}^{p,q}(f,A,F) = \sup\bigg\{\int_X (\mathcal{M}_{\mc T} \phi)^p \ud \mu :\phi\geq0, \phi \in  L^p(X,\mu),\\ \int_X\phi \ud \mu=f, \int_X\phi^q \ud \mu=A, \int_X\phi^p \ud \mu=F, \bigg\},
\end{multline}
where $1<q<p,$ and for $f, A, F$ we have $f^q < A < F^{\frac{q}{p}}.$ The new integral variable makes 
the problem considerably more difficult and thus makes it interesting to compute this Bellman function 
in a sub-domain of its original domain. We state our result in Theorem \ref{thm:1} below. The proof is an application of techniques presented in \cite{4}. So, with $\omega_p$ as defined above (and $\omega_q$ in an analogous way), in Section 3 we prove the following.
\begin{cthm}\label{thm:1} 
For $f, A$ such that $0<f^q<A$ and $F=F(f,A)$ satisfying $\omega_p(\frac{f^p}{F})=\omega_q(\frac{f^q}{A}),$ we have that
\begin{equation}\label{eq:1p10}
B_{\mc T}^{p,q}(f,A,F)=\omega_q\bigg(\frac{f^q}{A}\bigg)^pF.
\end{equation} 
\end{cthm}

Regarding the general case now, in Section 4 we provide an upper bound for \eqref{eq:1p9} (Lemma \ref{lem:4p2}). As a first step towards this, we prove an inequality satisfied by the corresponding maximal operator (Lemma \ref{lem:4p1}) which can be interpreted as the basic inequality proved in \cite{3}, for the case $\beta=0.$

We remark here that there are several problems in Harmonic Analysis were Bellman functions arise.
Such problems (including the dyadic Carleson imbedding theorem and weighted inequalities) are described
in \cite{8} (see also \cite{6}, \cite{7}) and also connections to Stochastic Optimal Control are 
provided, from which it follows that the corresponding Bellman functions satisfy certain nonlinear
second-order PDEs. The exact evaluation of a Bellman function is a difficult task which is connected
with the deeper structure of the corresponding Harmonic Analysis problem. Until now several Bellman 
functions have been computed (see \cite{1}, \cite{4}, \cite{5} \cite{6}, \cite{13}, \cite{14}, \cite{15}). The exact computation of \eqref{eq:1p6} has also been given in \cite{12} by L. Slavin,
A. Stokolos and V. Vasyunin, which linked the computation of it to solving certain PDEs of the 
Monge-Amp\`{e}re type, and in this way they obtained an alternative proof of the results in \cite{4} 
for the Bellman function related to the dyadic maximal operator. Also in \cite{16}, using the Monge-
Amp\`{e}re equation approach, a more general Bellman function than the one related to the Carleson 
imbedding theorem has been precisely evaluated thus generalizing the corresponding result in \cite{4}. 
It would be an interesting problem to discover if the Bellman function of three variables defined in 
\eqref{eq:1p9} can be computed using such PDE-based methods.

\section{Preliminaries}\label{sec:2}

In this section we present the background we need from \cite{4}, that will be used in all that follows.

Let $(X,\mu)$ be a non-atomic probability space. Two measurable subsets A, B of $X$ will be called
almost disjoint if $\mu(A\cap B)=0$.
\begin{definition}\label{definition:2p1}
A set $\mathcal{T}$ of measurable subsets of $X$ will be called a tree if the following conditions are
satisfied:
\begin{enumerate}
\item[(i)] $X \in \mathcal{T}$ and for every $I \in \mathcal{T}$ we have $\mu(I)>0$.
\item[(ii)]For every $I\in \mathcal{T}$ there corresponds a finite or countable subset
$\mathcal{C}(I) \subseteq \mathcal{T}$ containing at least two elements such that:
\begin{itemize}
\item[(a)]the elements of $\mathcal{C}(I)$ are pairwise almost disjoint subsets of $I$,
\item[(b)]$I=\bigcup \mathcal{C}(I)$.
\end{itemize}
\item[(iii)]$\mathcal{T}= \bigcup_{m\geq 0}\mathcal{T}_{m}$ where $\mathcal{T}_{(0)}=\{X\}$ and
$\mathcal{T}_{(m+1)}=\bigcup_{I\in \mathcal{T}_{(m)}} \mathcal{C}(I)$.
\item[(iv)]We have $\lim_{m \to \infty} \sup_{I \in \mathcal{T}_{(m)}}\mu(I)=0$.
\end{enumerate}
\end{definition}

By induction it can be seen that each family $\mathcal{T}_{(m)}$ consists of pairwise almost disjoint
sets whose union is $X$. Moreover if $x \in X\setminus E(\mathcal{T})$ then for each $m$ there exists
exactly
one $I_{m}(x)$ in $\mathcal{T}_{(m)}$ containing x. For every $m>0$ there is a
$J \in \mathcal{T}_{(m-1)}$ such that $I_{m}(x) \in \mathcal{C}(J)$. Then, since $x \in J,$ we must have
$J=I_{m-1}(x)$, because $x$ does not belong to $E(\mathcal{T})$. Hence the set $\mathcal{A}=\{I \in \mathcal{T}:x \in I\}$ forms a chain
$I_{0}(x)=X\varsupsetneq I_{1}(x)\varsupsetneq \dots$ with $I_{m} \in \mathcal{C}(I_{m-1}(x))$
for every $m>0$. From this remark it follows that if $I, J \in \mathcal{T}$ and
$I \cap J \cap (X\setminus E(\mathcal{T}))$ is non-empty, then $I\subseteq J$ or $J \subseteq I$. In
particular for any $I, J \in \mathcal{T}$, either $\mu(I\cap J)=0$ or one of them is contained in the
other. \\

We will also need the following.
\begin{lemma}\label{lem:2p1}
For every $I \in \mathcal{T}$ and every $a$ such that $0<a<1$ there exists a subfamily
$\mathcal{F}(I)\subseteq \mathcal{T}$ consisting of pairwise almost disjoint subsets of $I$ such that
\begin{equation}\label{eq:2p2}
\mu \left(\bigcup_{J \in \mathcal{F}(I)}J \right) = \sum_{J \in \mathcal{F}(I)}\mu(J)=
(1-a)\mu(I).
\end{equation}
\end{lemma}

The last thing we will use is the following Lemma. Let $\omega_q$ be as defined in the Introduction (note that the subscript is q instead
of p).
\begin{lemma}\label{lem:2p2}
Let $q>1$ and $\tau \in (0, 1]$ be fixed. Then for every $\alpha$ with $0<\alpha<1$ the equation
\begin{equation}\label{eq:2p9}
-(z-\alpha)^q + (1-\alpha)^{q-1}z^q=\tau\alpha(1-\alpha)^{q-1}
\end{equation}
has a unique solution $z=z(\alpha, \tau) \in [1, \infty)$ and moreover
\begin{equation}\label{eq:2p10}
\lim_{\alpha \to 0^+}z(\alpha, \tau)=\omega_q(\tau).
\end{equation}
\end{lemma}

\section{Computation of the Bellman function on a sub-domain}\label{sec:3}

Following \cite{4}, we choose $\alpha$ with $0<\alpha <1$ and using Lemma
\ref{lem:2p1}, for every $I \in \mathcal{T}$ we choose a family $\mathcal{F}(I)\subseteq \mathcal{T}$
of pairwise almost disjoint subsets of $I$ such that
\begin{equation}\label{eq:3p1}
\mu \left(\bigcup_{J \in \mathcal{F}(I)}J \right) = \sum_{J \in \mathcal{F}(I)}\mu(J)=
(1-a)\mu(I).
\end{equation}
Then we define $\mathcal{S}=\mc S_\alpha$ to be the smallest subset of $ \mathcal{T}$ such that $X \in
\mathcal{S}$ and for every $I \in \mathcal{S}$, $\mathcal{F}(I) \subseteq \mathcal{S}$ and
the correspondence $I \to I^*$ with respect to this $\mathcal{S},$ by setting $J^*=I \in \mathcal{S}$
if and only if $J \in \mathcal{F}(I)$ and so writing
\begin{equation}\label{eq:3p2}
A_I =I \setminus \bigcup_{J \in \mathcal{S}:\; J^*=I}J,
\end{equation}
we have $a_I=\mu(A_I)=\mu (I) - \sum_{J \in \mathcal{S}:\; J^*=I}\mu (J)=\alpha\mu(I)$ for every
$I \in \mathcal{S}$.
Also it is easy to see that
\begin{equation}\label{eq:3p3}
\mathcal{S}=\bigcup_{m \geq 0}\mathcal{S}_{(m)}, \ \text{where} \ \mathcal{S}_{(0)}=\{X\} \ \text{and} \
\mathcal{S}_{(m+1)}=\bigcup_{I \in \mathcal{S}_{(m)}}\mathcal{F}(I).
\end{equation}
The rank $r(I)$ of any $I \in \mathcal{S}$ is now defined to be the unique integer $m$ such that
$I \in \mathcal{S}_{(m)}$.

For $\alpha \in (0,1)$ and the family $\mathcal{S}=\mc S_\alpha$ defined as above,
we set
\begin{equation}\label{eq:3p4}
x_I =\lambda\gamma^{r(I)}\mu(I)^{1/q}
\end{equation}
where $\lambda=f\alpha^{-1/\acute{q}}(1-\gamma(1-\alpha)), \; 1/q +1/\acute{q} = 1,  \; \gamma=\frac{\beta+1}{\beta+1-\beta\alpha}, \; \beta >0.$
For every $I \in \mathcal{S}$ and every $m \geq 0$ we write
\begin{equation}\label{eq:3p5}
b_m(I)= \sum_{\substack{\mathcal{S}\ni J \subseteq I \\ r(J)=r(I)+m}}\mu(J)
\end{equation}
and observing that
\begin{equation}\label{eq:3p6}
b_{m+1}(I)=\sum_{\substack{\mathcal{S}\ni J \subseteq I, \\ r(J)=r(I)+m}}
\sum_{L \in \mathcal{F}(J)}\mu(L)=(1-\alpha)b_m(I)
\end{equation}
we get
\begin{equation}\label{eq:3p7}
b_m(I)=(1-\alpha)^m \mu(I).
\end{equation}

Hence 
\begin{align}\label{eq:3p8}
\sum_{I \in \mathcal{S}_\alpha}x_I^q 
&=\lambda^q\sum_{m\geq 0}\sum_{I \in \mc{S}_{(m)}}\gamma^{mq}\mu(I)
=\lambda^q\sum_{m\geq 0}\gamma^{mq}b_m(X)\nonumber \\
&=\lambda^q\sum_{m\geq 0}[\gamma^q(1-\alpha)]^m
=\frac{\lambda^q}{1-\gamma^q(1-\alpha)},
\end{align} 
assuming that $\gamma^q(1-\alpha)<1$. Additionally,
\begin{align}\label{eq:3p9}
a_I^{1/q} y_I & = \frac{1}{\mu(I)}\sum_{\substack{J \in \mathcal{S} \\ J
\subseteq I}}a_I^{1/q}a_J^{1/\acute{q}}x_J \nonumber \\
&=\frac{\lambda}{\mu(I)}\sum_{\substack{J \in \mathcal{S} \\ J
\subseteq I}}(\alpha \mu(I))^{1/q}(\alpha \mu(J))^{1/\acute{q}}\gamma^{r(J)}\mu(J)^{1/q}\nonumber \\
&=\frac{\alpha\lambda}{\mu(I)^{1-\frac{1}{q}}}\sum_{m\geq 0}\gamma^{m+r(I)}
\sum_{\substack{\mathcal{S}\ni J \subseteq I \\ r(J)=r(I)+m}}\mu(J)\nonumber \\
&=\frac{\alpha\lambda}{\mu(I)^{1-\frac{1}{q}}}\gamma^{r(I)}\sum_{m\geq 0}
\gamma^m\mu(I)(1-\alpha)^m 
=\frac{\alpha}{(1-\gamma(1-\alpha))}x_I,
\end{align}
where the $y_I$'s are defined by the first equality above.

Now we consider the function
\begin{equation}\label{eq:3p10}
\varphi_\alpha=\sum_{I \in \mathcal{S}}\frac{x_I}{a_I^{1/q}}\chi_{A_I},
\end{equation}
$I \in \mathcal{S}$. It is easy to see that
\begin{equation}\label{eq:3p11}
\int_X \phi_\alpha \ud \mu =f 
\end{equation}
and writing $z=\beta+1-\beta\alpha$ and thus $\gamma=\frac{\beta+1}  {\beta+1-\beta\alpha}=\frac{z-\alpha}{z(1-\alpha)},$ 
\begin{equation}\label{eq:3p12}
\int_X\phi_\alpha^q\,\mr d\mu=\sum_{I \in \mathcal{S}}x_I^q = \frac{f^q\alpha(1-\alpha)^{q-1}}{-(z-\alpha)^q + (1-\alpha)^{q-1}z^q}.
\end{equation}
    
We are now in a position to prove Theorem \ref{thm:1}.
Let $\phi_\alpha$ be as in \eqref{eq:3p10} and $z=z(\alpha,\frac{f^q}{A})$ the solution provided by Lemma \ref{lem:2p2} which makes $\int_X\phi_\alpha^q\,\mr d\mu=A$ (and also
$\gamma^q(1-\alpha)=\frac{(z-\alpha)^q}{z^q(1-\alpha)^{q-1}}=1-\frac{\alpha f^q}{z^qA}<1$).
Set $F(\alpha)=\int_X \phi_\alpha^p \ud \mu = \sum_{I \in \mathcal{S}_\alpha}a_I^{-\frac{p-q}{q}} x_I^p.$ Since
\begin{equation*}
\mt \phi_\alpha \geq \sum_{I\in\mathcal{S}}(\frac{1}{\mu(I)}\int_I\phi_\alpha \ud\mu) \chi_{A_I} =
\sum_{I\in\mathcal{S}}\frac{1}{\mu(I)}\big(\sum_{\substack{J\in \mathcal{S} \\J \subseteq I}}
a_J^{1/\acute{q}}x_J \big)\chi_{A_I},
\end{equation*}
from \eqref{eq:3p9} we have
\begin{equation}\label{eq:3p13}
\int_X (\mt \phi_\alpha)^p \ud \mu \geq \sum_{I\in\mathcal{S}}a_I y_{I}^p = z^{p}F(\alpha).
\end{equation}
From Lemma \ref{lem:2p2}, $z \to \omega_q(\frac{f^q}{A})$ and, arguing as in \eqref{eq:3p8}, $$F(\alpha)
=\frac{f^p\alpha(1-\alpha)^{p-1}}{-(z-\alpha)^p+(1-\alpha)^{p-1}z^p}
\to \frac{f^p}{H_p(\omega_q(\frac{f^q}{A}))}=F(f,A)$$ as $\alpha \to 0^+.$
So taking limits in \eqref{eq:3p13}, we conclude that
\begin{equation}\label{eq:3p14}
B_{\mc T}^{p,q}(f,A,F(f,A))\geq \omega_q(\frac{f^q}{A})^p F(f,A).
\end{equation}
To see that the converse inequality also holds, we use \eqref{eq:1p7}, our hypothesis
on the relation between $f, A, F,$ and the inequality
$$B_{\mc T}^{p,q}(f,A,F) \leq B_{\mc T}^{(p)}(f,F),$$ which holds by the corresponding definitions.
So finally, 
\begin{equation}\label{eq:3p15}
B_{\mc T}^{p,q}(f,A,F(f,A))=\omega_q(\frac{f^q}{A})^p F(f,A).
\end{equation}
which is what we aimed to prove.

\section{Upper bound for \eqref{eq:1p9} }\label{sec:4}

We first prove an inequality (Lemma \ref{lem:4p1}) satisfied by the maximal operator defined in 
\eqref{eq:1p5}. For this proof we use a variation of the approach that arises in \cite{9}, as can be 
seen below. An application of this, in case one also considers the $L^p$-norm of $\phi$ as given, provides us with an upper bound to \eqref{eq:1p9}. This is the content of Lemma \ref{lem:4p2}.

\begin{lemma}\label{lem:4p1}
If $I=\int_X (\mathcal{M}_{\mathcal{T}}\phi)^p \ud \mu$ and $\int_X\phi\,\mr d\mu=f,$ $\int_X\phi^q\,\mr d\mu=A$ then,
\begin{equation*}
I \leq f^p- \frac{p}{p-q}f^{p-q}A +
\frac{p}{p-q}\int_X (\mathcal{M}_{\mathcal{T}}\phi)^{p-q}\phi^q \ud \mu.
\end{equation*}
\end{lemma}

\begin{proof}
We have
\begin{align}\label{eq:4p1}
I&=\int_{\lambda=0}^{\infty}p\lambda^{p-1}\mu(\{\mathcal{M}_{\mathcal{T}}\phi\geq \lambda\})\ud\lambda
\nonumber \\&=
\int_{\lambda=0}^{f}p\lambda^{p-1}\mu(\{\mathcal{M}_{\mathcal{T}}\phi\geq \lambda\})\ud\lambda +
\int_{\lambda=f}^{\infty}p\lambda^{p-1}\mu(\{\mathcal{M}_{\mathcal{T}}\phi\geq \lambda\})\ud\lambda
\nonumber \\ &=I_1+I_2.
\end{align}
Then, since $\{\mathcal{M}_{\mathcal{T}}\phi\geq \lambda\}=X$, for every $\lambda<f$ (because
$\mathcal{M}_{\mathcal{T}}\phi(x)\geq f$, for every $x \in X$), we have
\begin{equation}\label{eq:4p2}
I_1=\int_{\lambda=0}^{f}p\lambda^{p-1}\ud \lambda =f^p
\end{equation}
and
\begin{equation}\label{eq:4p3}
I_2=\int_{\lambda=f}^{\infty}p\lambda^{p-1}\mu(\{\mathcal{M}_{\mathcal{T}}\phi\geq \lambda\})\ud\lambda.
\end{equation}
For $\lambda>f$ set $E_\lambda=\{\mathcal{M}_{\mathcal{T}}\phi\geq \lambda\}=\cup I_j,$ where
$I_j \in \mathcal{T}$ maximal with respect to the condition
$\frac{1}{\mu(I_j)}\int_{I_j}\phi\ud\mu \geq \lambda.$ Since by
maximality the $I_j$'s are pairwise disjoint,
$\mu(E_\lambda)=\sum\mu(I_j) \leq\sum_j\frac{1}{\lambda}\int_{I_j}\phi\ud\mu= \frac{1}{\lambda}
\int_{\cup I_j}\phi\ud\mu=\frac{1}{\lambda} \int_{E_\lambda}\phi\ud\mu.$ So
\begin{equation}\label{eq:4p4}
\mu(E_\lambda)\leq \frac{1}{\lambda} \int_{E_\lambda}\phi\ud\mu,
\end{equation}
which is the general weak-type $(1,1)$ inequality for the dyadic maximal operator. From
\eqref{eq:4p4} we conclude that $\frac{1}{\mu(E_\lambda)} \int_{E_\lambda}\phi\ud\mu \geq \lambda$ and
consequently, for $1<q<p,$ by using H\"{o}lder's inequality, we obtain
$\lambda^q \leq (\frac{1}{\mu(E_\lambda)} \int_{E_\lambda}\phi\ud\mu)^q \leq
\frac{1}{\mu(E_\lambda)} \int_{E_\lambda}\phi^q\ud\mu$ and so
\begin{equation}\label{eq:4p5}
\mu(E_\lambda)\leq\frac{1}{\lambda^q} \int_{E_\lambda}\phi^q\ud\mu.
\end{equation}
Inserting \eqref{eq:4p5} in \eqref{eq:4p3}, we get
\begin{align}\label{eq:4p6}
I_2 &\leq \int_{\lambda=f}^{\infty}p\lambda^{p-1}\frac{1}{\lambda^q} \int_{\{\mathcal{M}_{\mathcal{T}}\phi\geq \lambda\}}\phi^q\ud\mu\ud\lambda\nonumber\\&=
\int_X p\phi(x)^q\big(\int_{\lambda=f}^{\mathcal{M}_{\mathcal{T}}\phi(x)}\lambda^{p-q-1}\ud\lambda\big)\ud\mu(x)=\nonumber\\&=
\frac{p}{p-q}\int_X \phi(x)^q [\lambda^{p-q}]_{\lambda=f}^
{\mathcal{M}_{\mathcal{T}}\phi(x)}\ud\mu(x)\nonumber\\&
=\frac{p}{p-q}\int_X (\mathcal{M}_{\mathcal{T}}\phi)^{p-q}\phi^q \ud \mu
-\frac{p}{p-q}f^{p-q}\int_X\phi^q\ud\mu \nonumber\\&=
\frac{p}{p-q}\int_X (\mathcal{M}_{\mathcal{T}}\phi)^{p-q}\phi^q \ud \mu
-\frac{p}{p-q}f^{p-q}A.
\end{align}
From \eqref{eq:4p6} and \eqref{eq:4p1}, \eqref{eq:4p2} we finally get
\begin{equation}\label{eq:4p7}
I=\int_X (\mathcal{M}_{\mathcal{T}}\phi)^p \ud \mu \leq f^p- \frac{p}{p-q}f^{p-q}A +
\frac{p}{p-q}\int_X (\mathcal{M}_{\mathcal{T}}\phi)^{p-q}\phi^q \ud \mu
\end{equation}
which is the desired inequality.
\end{proof}

\begin{lemma}\label{lem:4p2}
If $\int_X\phi\,\mr d\mu=f,$ $\int_X\phi^q\,\mr d\mu=A$ and $\int_X\phi^p\,\mr d\mu=F,$ then
\begin{equation}\label{eq:4p8}
I=\int_X (\mathcal{M}_{\mathcal{T}}\phi)^p \ud \mu \leq F h^{-1}(\frac{p f^{p-q}A-(p-q)f^p}{F})^p,
\end{equation}
where $h:[1,\infty) \to (-\infty,q]$ with $h(t)=pt^{p-q}-(p-q)t^p.$
\end{lemma}
\begin{proof}
From \eqref{eq:4p7}, after using H\"older inequality for the integral on the right hand side,
we have
$$I=\int_X (\mathcal{M}_{\mathcal{T}}\phi)^p \ud \mu \leq f^p- \frac{p}{p-q}f^{p-q}A +
\frac{p}{p-q}I^\frac{p-q}{p}F^\frac{q}{p},$$
which, if we divide by F becomes
$$(p-q)\frac{I}{F} \leq \frac{(p-q)f^p-pf^{p-q}A}{F}+p\Big(\frac{I}{F}\Big)^\frac{p-q}{p}$$ and this gives
$$p\Big(\frac{I}{F}\Big)^\frac{p-q}{p}-(p-q)\Big(\frac{I}{F}\Big) \geq k(f,A,F),$$
where $k(f,A,F)=(pf^{p-q}A-(p-q)f^p)/F.$ So
$$p\Big[\Big(\frac{I}{F}\Big)^\frac{1}{p}\Big]^{p-q} - (p-q)\Big[\Big(\frac{I}{F}\Big)^\frac{1}{p}\Big]^p \geq k(f,A,F) \Rightarrow$$
\begin{equation}\label{eq:4p9}
h\Big[\Big(\frac{I}{F}\Big)^\frac{1}{p}\Big] \geq k(f,A,F).
\end{equation}
We consider now the function $h(t)=pt^{p-q}-(p-q)t^p,$ $t>0.$ We observe that
\begin{equation*}
h'(t)=p(p-q)t^{p-q-1}-p(p-q)t^{p-1}=p(p-q)t^{p-q-1}(1-t^q).
\end{equation*}
So $h$ is strictly decreasing on $[1,\infty)$, and strictly increasing on $(0,1]$. Additionally, $h(1)=q$ while
$\lim_{t\to \infty}h(t)=-\infty.$ Thus $h(t)\leq q$, for any $t>0.$ At the same time
\begin{align*}
k(f,A,F)&=\frac{pf^{p-q}A-(p-q)f^p}{F}= \frac{-(p-q)\big(\frac{f}{A^{1/q}}\big)^p +
p\big(\frac{f}{A^{1/q}}\big)^{p-q}}{F}   A^{\frac{p}{q}}\nonumber\\&=
\frac{-(p-q)\tau^p + p\tau^{p-q}}{F} A^{\frac{p}{q}}= \Big[\frac{h(\tau)}{F}\Big]A^{\frac{p}{q}}
\leq q\frac{A^{\frac{p}{q}}}{F}\leq q,
\end{align*}
where $\tau=\frac{f}{A^{1/q}}$.
So $$q\geq k(f,A,F)>0$$ and \eqref{eq:4p9} yields
\begin{equation*}
\Big(\frac{I}{F}\Big)^\frac{1}{p} \leq h^{-1}(k(f,A,F)),
\end{equation*}
which is what we wanted.
\end{proof}

Anastasios D. Delis, Eleftherios N. Nikolidakis, National and Kapodistrian University of Athens, Department of Mathematics, Panepistimioupolis, Zografou 157 84,
Athens, Greece.

\end{document}